\documentclass[12pt]{amsart}

\usepackage{amsmath,amsthm,amsfonts,latexsym,amssymb,amscd,color}
\usepackage{graphics,multimedia,tikz}
\usepackage{comment}

\numberwithin{equation}{section}
\theoremstyle{plain} 
\newtheorem{thm}{Theorem}[section]

\newtheorem{lm}[thm]{Lemma}
\newtheorem{prp}[thm]{Proposition}

\theoremstyle{definition}
\newtheorem{df}[thm]{Definition}

\newtheorem{remark}[thm]{Remark}
\newtheorem{case}{Case}

\newcommand{\A}{\mathbf{A}}

\newcommand{\dmn}{\mathrm{dmn}}
\newcommand{\con}{\mathrm{Con}}

\newcommand{\vp}{\mathbf{p}}
\newcommand{\vq}{\mathbf{q}}

\newcommand{\vx}{\mathbf{x}}
\newcommand{\vy}{\mathbf{y}}

\newcommand{\termcube}[8]{
  \begin{tikzpicture}
\filldraw[black] (0,0) circle (2pt) node[anchor=east] {$ #1 $};
\filldraw[black] (0,2) circle (2pt) node[anchor=east] {$ #2 $};
\filldraw[black] (2,0) circle (2pt) node[anchor=west] {$ \; #3 $};
\filldraw[black] (2,2) circle (2pt) node[anchor=west] {$ \; #4 $};

\draw [very thick] (0,0) -- (0,2);
\draw [very thick] (2,0) -- (2,2);
\draw (0,0) -- (2,0);
\draw (0,2) -- (2,2);

\filldraw[black] (1,-.5) circle (2pt) node[anchor=east] {$ #5 \; $};
\filldraw[black] (1,1.5) circle (2pt) node[anchor=east] {$ #6 \; $};
\filldraw[black] (3,-.5) circle (2pt) node[anchor=west] {$ #7 $};
\filldraw[black] (3,1.5) circle (2pt) node[anchor=west] {$ #8 $};

\draw [very thick] (1,-.5) -- (1,1.5);
\draw [very thick, dashed] (3,-.5) -- (3,1.5);
\draw (1,-.5) -- (3,-.5);
\draw (1,1.5) -- (3,1.5);

\draw (0,2) -- (1,1.5);
\draw (2,2) -- (3,1.5);
\draw (0,0) -- (1,-.5);
\draw (2,0) -- (3,-.5);

\end{tikzpicture}
}

\newcommand{\tcube}[8]{
  \begin{tikzpicture}
\filldraw[black] (0,0) circle (2pt) node[anchor=east] {$ #1 $};
\filldraw[black] (0,2) circle (2pt) node[anchor=east] {$ #2 $};
\filldraw[black] (2,0) circle (2pt) node[anchor=west] {$ \; #3 $};
\filldraw[black] (2,2) circle (2pt) node[anchor=west] {$ \; #4 $};

\draw (0,0) -- (0,2);
\draw (2,0) -- (2,2);
\draw (0,0) -- (2,0);
\draw (0,2) -- (2,2);

\filldraw[black] (1,-.5) circle (2pt) node[anchor=east] {$ #5 \; $};
\filldraw[black] (1,1.5) circle (2pt) node[anchor=east] {$ #6 \; $};
\filldraw[black] (3,-.5) circle (2pt) node[anchor=west] {$ #7 $};
\filldraw[black] (3,1.5) circle (2pt) node[anchor=west] {$ #8 $};

\draw (1,-.5) -- (1,1.5);
\draw (3,-.5) -- (3,1.5);
\draw (1,-.5) -- (3,-.5);
\draw (1,1.5) -- (3,1.5);

\draw (0,2) -- (1,1.5);
\draw (2,2) -- (3,1.5);
\draw (0,0) -- (1,-.5);
\draw (2,0) -- (3,-.5);

\end{tikzpicture}
}

\newcommand{\termsquare}[4]{
\begin{tikzpicture}
\filldraw[black] (0,0) circle (2pt) node[anchor=east] {$ #1 $};
\filldraw[black] (0,2) circle (2pt) node[anchor=east] {$ #2 $};
\filldraw[black] (2,0) circle (2pt) node[anchor=west] {$ #3 $};
\filldraw[black] (2,2) circle (2pt) node[anchor=west] {$ #4 $};

\draw [very thick] (0,0) -- (0,2);
\draw [very thick, dashed] (2,0) -- (2,2);
\draw (0,0) -- (2,0);
\draw (0,2) -- (2,2);

\end{tikzpicture}
}

\newcommand{\tsquare}[4]{
\begin{tikzpicture}
\filldraw[black] (0,0) circle (2pt) node[anchor=east] {$ #1 $};
\filldraw[black] (0,2) circle (2pt) node[anchor=east] {$ #2 $};
\filldraw[black] (2,0) circle (2pt) node[anchor=west] {$ #3 $};
\filldraw[black] (2,2) circle (2pt) node[anchor=west] {$ #4 $};

\draw (0,0) -- (0,2);
\draw (2,0) -- (2,2);
\draw (0,0) -- (2,0);
\draw (0,2) -- (2,2);

\end{tikzpicture}
}

\begin{document}


\title[Central Series for Simple Algebras]{On the Descending Central Series of Higher Commutators for Simple Algebras}

\author[Steven Weinell]{Steven Weinell}
\email{steven.weinell@colorado.edu}
\address{Department of Mathematics\\
University of Colorado\\
Boulder, CO 80309-0395\\
USA}

\thanks{This material is based upon work supported by
the National Science Foundation grant no.\ DMS 1500254.}
\subjclass[2010]{Primary: 08A05; Secondary: 08A30}
\keywords{central series, higher commutator, nilpotent, simple, supernilpotent}

\begin{abstract}
This paper characterizes the potential behaviors of higher commutators in a simple algebra.
\end{abstract}

\maketitle





\section{Introduction}\label{intro}
In \cite{bulatov}, Bulatov defined a higher commutator for general algebraic structures. Using this higher commutator we may define \emph{the descending central series of higher commutators}:
	\begin{enumerate}
	\item $\theta_1 = 1_\A$
	\item $\theta_{m} = [\underbrace{1_\A, \dots, 1_\A}_{n \text{ many}}]$ for $m \geq 2$.
	\end{enumerate}
Higher commutators satisfy $[\alpha_1, \dots, \alpha_m, \alpha_{m+1}] \leq [\alpha_1, \dots, \alpha_m]$ for congruences $\alpha_1, \dots, \alpha_{m+1}$. So the descending central series of higher commutators is weakly descending. More precisely, the descending central series of higher commutators forms a weakly descending chain in the lattice of congruences of $\A$. We separate all weakly descending chains $(\theta_{m})_m$ in a two element lattice with $\theta_1 = 1$ into three types. We can have a weakly descending chain which never descends, so $\theta_m = 1$ for all $m$. We can have a chain which immediately descends, so $\theta_m = 0$ for all $m \geq 2$. And finally we can have the general case where there is some $n \geq 2$ with 
\[ 
\theta_1 = \theta_2 = \dots = \theta_n = 1_\A
\]
and 
\[
\theta_{n+1} = \theta_{n+2} = \dots = 0_\A
\]
This paper will establish that any of these possibilities may be represented as the descending central series of higher commutators for some algebra. Representing the first two possibilities is not difficult. Our main theorem will construct an algebra which represents the general case.

An algebra $\A$ is \emph{supernilpotent} if the descending central series of higher commutators has some $n$ such that $\theta_n = 0_\A$. The study of supernilpotence has driven much of the study of descending central series of higher commutators. In \cite{aichinger-mudrinski} Aichinger and Mudrinski established that in a Mal'cev algebra supernilpotence implies nilpotence. In \cite{kearnes-szendrei} Kearnes and Szendrei showed that this holds in any finite algebra. In \cite{moore-moorhead} it was shown by Moore and Moorhead that supernilpotence does not always imply nilpotence. A corollary of the theorem in this paper, when $n=2$, is that there exists a nonabelian simple algebra which is supernilpotent. Such an algebra cannot be nilpotent.




\section{Preliminaries}\label{preliminaries}

In this paper the algebra that we analyze, $\A$, will be simple, i.e. our algebra will only have two congruences, $1_\A$ and $0_\A$. We will take advantage of this fact to allow ourselves to simplify our arguments in two ways. First, we recall that in the congruence lattice of $\A$, $\con(\A)$, we have $[\theta_1, \theta_2, \dots, \theta_n] \leq \theta_i$ for all $1 \leq i \leq n$. So in a simple algebra, we have
\[
\text{If there exists an } i \text{ with } \theta_i = 0_\A \text{ then } [\theta_1, \theta_2, \dots, \theta_n] = 0_\A 
\]
Thus we will only be concerned with higher commutators of all $1_\A$'s. Second, we will simplify notation from what is normally needed to discuss higher commutators in a general algebra. This chapter is devoted to describing our simplified notation.

For a more general exposition on centralizers and higher commutators see, for example, \cite{freese-mckenzie}. 

\begin{remark} 
It will be convenient to now state standardized notation for tuples in the rest of this paper. A tuple will be represented by a bold letter, e.g. $\vp, \vx, \vx_1$. The components of a tuple will be represented by non-bold letters which match the tuple name and are subscripted by non-zero natural numbers. If a letter has two subscripts, we will separate them by commas. So given tuples $\vp, \vx, \vx_1$ of lengths $l$, $m$, and $k$, respectively, our convention will be $\vp = (p_1, p_2, \dots, p_l)$, $\vx = (x_1, x_2, \dots, x_m)$, $\vx_1 = (x_{1, 1}, x_{1, 2}, \dots x_{1, k})$. 
\end{remark}

\begin{df} Let $\tau(\vx_1, \dots, \vx_n)$ be a term in the language of $\A$. We will write $\tau^\A$ to be the term operation on $\A$ obtained by replacing each function symbol with its interpretation in $\A$. The \emph{$n$-term cube} for $\tau^\A$ on tuples $\vp_1^0, \vp_1^1, \vp_2^0, \vp_2^1,\dots, \vp_n^0, \vp_n^1$ is the $2^n$ tuple 
\[
C_{\tau^\A}^n(\vp_1^0, \vp_1^1; \vp_2^0, \vp_2^1;\dots; \vp_n^0, \vp_n^1) = (r_1, r_2, r_3, \dots, r_{2^n})
\]
with 
\[
r_i = \tau^\A(\vp_1^{i_1}, \vp_2^{i_2}, \dots, \vp_n^{i_n}) \qquad \text{ where } i-1 = \sum_{j=1}^n i_j 2^{j-1}
\]
Note $i_j$ is the $j^\text{th}$ digit of the number $i-1$ written in binary. We call $r_i$ the \emph{$i^\text{th}$ vertex} of the $n$-term cube. We will write $C_{\tau^\A}^n(\vp_i, \vq_i)$ for the term cube where $\vp_i^0 = \vp_i$ and $\vp_i^1 = \vq_i$ for $1 \leq i \leq n$. I.e.
\[
C_{\tau^\A}^n(\vp_i, \vq_i) = C_{\tau^\A}^n(\vp_1, \vq_1; \vp_2, \vq_2;\dots; \vp_n, \vq_n)
\]
We will sometimes write $C_{\tau^\A}^n$ as the $n$-term cube for $\tau^\A$ if the tuples are understood from context.
We will call $C_{\tau^\A}^2(\vp_i, \vq_i)$ the \emph{term square} for $\tau^\A$ on $\vp_1, \vq_1, \vp_2, \vq_2$ and write $S_{\tau^\A}(\vp_i, \vq_i)$. We will display 
\[
S_{\tau^\A}(\vp_i, \vq_i) = (r_1, r_2, r_3, r_4)
\]
pictorially as in Figure \ref{squaredef}.
\begin{figure}[h]
\tsquare{\tau^\A(\vp_1, \vp_2) = r_1}{\tau^\A(\vp_1, \vq_2) = r_2}{r_3 = \tau^\A(\vq_1, \vp_2)}{r_4 = \tau^\A(\vq_1, \vq_2)}
\caption{A pictorial representation of $S_{\tau^\A}(\vp_i, \vq_i)$.}
\label{squaredef}
\end{figure}

We will call $C_{\tau^\A}^3(\vp_i, \vq_i)$ the \emph{term cube} for $\tau^\A$ on $\vp_1, \vq_1, \vp_2, \vq_2, \vp_3, \vq_3$ and write $C_{\tau^\A}(\vp_i, \vq_i)$. We will display 
\[
C_{\tau^\A}(\vp_i, \vq_i) = (r_1, r_2, r_3, r_4, r_5, r_6, r_7, r_8)
\]
pictorially as in Figure \ref{cubedef}.
\begin{figure}[h]
\tcube
{\tau^\A(\vp_1, \vp_2, \vp_3) = r_1}
{\tau^\A(\vp_1, \vp_2, \vq_3) = r_2}
{r_3 = \; \tau^\A(\vp_1, \vq_2, \vp_3)}
{r_4 = \tau^\A(\vp_1, \vq_2, \vq_3)}
{\tau^\A(\vq_1, \vp_2, \vp_3) = r_5}
{\tau^\A(\vq_1, \vp_2, \vq_3) \, = r_6}
{r_7 = \tau^\A(\vq_1, \vq_2, \vp_3)}
{r_8 = \tau^\A(\vq_1, \vq_2, \vq_3)}
\caption{A pictorial representation of $C_{\tau^\A}(\vp_i, \vq_i)$.}
\label{cubedef}
\end{figure}
\end{df}

\begin{df} We say that an algebra $\A$ \emph{fails the $n$-dimensional term condition} if there exists a term in the language of $\A$, $\tau(\vx_1, \dots, \vx_n)$, and tuples $\vp_1, \dots, \vp_n$ and $\vq_1, \dots, \vq_n$ such that 
\[
C_{\tau^\A}^n(\vp_i, \vq_i) = (r_1, r_2, r_3, \dots, r_{2^n})
\]
has
\[
r_{2i-1} = r_{2i} \qquad \text{ for all } 1 \leq i \leq 2^{n-1}-1
\]
and
\[
r_{2^n - 1} \neq r_{2^n}
\]
We say the term $\tau$ above \emph{witnesses} the failure of the $n$-dimensional term condition. An algebra \emph{satisfies} the $n$-dimensional term condition if it does not fail the $n$-dimensional term condition. The 2-dimensional and 3-dimensional term condition may be displayed pictorially as in Figure \ref{23termcondition}. If there is a term operation producing equality along the bold vertical lines and inequality along the dashed vertical, then the algebra fails the $n$-dimensional term condition for $n = 2$ or $3$, respectively. We call the dashed vertical edge the \emph{critical edge} for this reason.
\begin{figure}[h]
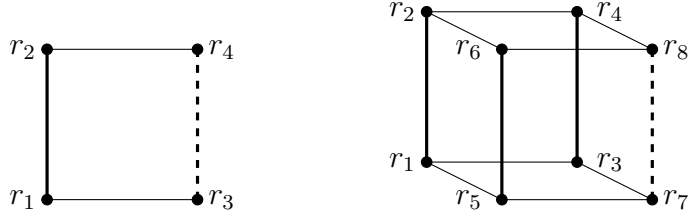

\termsquare{r_1}{r_2}{r_3}{r_4} \hskip1.5cm \termcube{r_1}{r_2}{r_3}{r_4}{r_5}{r_6}{r_7}{r_8}
\caption{Pictorial representations of the 2-term and 3-term conditions.}
\label{23termcondition}
\end{figure}

\end{df}

\begin{remark} The $n$ commutator $[\underbrace{1_\A, 1_\A, \dots, 1_\A}_{n \text{ many}}] = 0_\A$ if and only if $\A$ satisfies the $n$-dimensional term condition.
\end{remark}




\section{The Construction}\label{construction}

Fix $n \geq 2$ for the remainder of this paper. Though not explicit in our notation, the constructed algebra $\A$ will depend on this fixed number $n$. To define the algebra $\A$ we will first need to recursively define an $\omega$-sequence of partial algebras.

\begin{df} Let
\[
B = \{\, a_{i, j} \mid 1 \leq i \leq n \text{ and } j \in \omega \,\} \cup \{\, b_{i, j} \mid 1 \leq i \leq n \text{ and } j \in \omega \,\}
\]
For convenience we set $a_i = a_{i, 0}$ and $b_i = b_{i, 0}$ for $1 \leq i \leq n$ and 
\[
C = \{\, a_1, \dots, a_n, b_1, \dots, b_n\}
\]
We now define the first partial algebra $\A_0$. The universe of $\A_0$ will be 
\[
A_0 = B \cup \{\, d_i \mid 1 \leq i \leq 2^{n-1}+1 \,\} \cup \{\, c \,\}
\]
The language of $\A_0$ will have one $n$-ary function symbol $f_0$. $\A_0$ will interpret this function with domain
\[
\dmn(f_0^{\A_0}) = \{\, (x_1, x_2, \dots, x_n) \mid x_i \in \{\, a_i, b_i \,\} \,\} 
\]
so that the term $f_0(x_1, \dots, x_n)$ has $n$-term cube
\[
C_{f_0^{\A_0}}^n(a_i, b_i) = (d_1, d_1, d_2, d_2, \dots, d_{2^{n-1}-1}, d_{2^{n-1}-1}, d_{2^{n-1}}, d_{2^{n-1}+1})
\]
More precisely, $f_0^{\A_0}$ is defined as follows. 
\begin{align*} 
f_0^{\A_0}(a_1, a_2, a_3, \dots, a_{n-2}, a_{n-1}, a_n) &= d_1 \\ 
f_0(^{\A_0}a_1, a_2, a_3, \dots, a_{n-2}, a_{n-1}, b_n) &= d_1 \\ 
f_0^{\A_0}(a_1, a_2, a_3, \dots, a_{n-2}, b_{n-1}, a_n) &= d_2 \\ 
f_0^{\A_0}(a_1, a_2, a_3, \dots, a_{n-2}, b_{n-1}, b_n) &= d_2 \\ 
f_0^{\A_0}(a_1, a_2, a_3, \dots, b_{n-2}, a_{n-1}, a_n) &= d_3 \\ 
f_0^{\A_0}(a_1, a_2, a_3, \dots, b_{n-2}, a_{n-1}, b_n) &= d_3 \\ 
& \; \vdots \\
f_0^{\A_0}(b_1, b_2, b_3, \dots, b_{n-2}, a_{n-1}, a_n) &= d_{2^{n-1}-1} \\ 
f_0^{\A_0}(b_1, b_2, b_3, \dots, b_{n-2}, a_{n-1}, b_n) &= d_{2^{n-1}-1} \\ 
f_0^{\A_0}(b_1, b_2, b_3, \dots, b_{n-2}, b_{n-1}, a_n) &= d_{2^{n-1}} \\ 
f_0^{\A_0}(b_1, b_2, b_3, \dots, b_{n-2}, b_{n-1}, b_n) &= d_{2^{n-1}+1} \\ 
\end{align*}
Next we define $\A_{i+1}$ for $i \in \omega$. The universe of $\A_{i+1}$ will be
\[
A_{i+1} = A_i \cup (( A_i^n \setminus \dmn(f_i) ) \times \{\, i \,\} )
\]
The language of $\A_{i+1}$ will have one $n$-ary function symbol $f_{i+1}$. $\A_{i+1}$ will interpret this function with domain
\[
\dmn(f_{i+1}^{\A_{i+1}}) = A_i^n
\]
as
\[
f_{i+1}^{\A_{i+1}}(\vx) = 
\begin{cases}
f_i^{\A_{i}}(\vx) & \text{ for } \vx \in \dmn(f_i^{\A_{i}}) \\
(\vx, i) & \text{ otherwise }
\end{cases}
\]

We are now ready to define our desired algebra $\A$. The universe of $\A$ will be $A = \bigcup_{i \in \omega} A_i$. The language of $\A$ will have the following set of function symbols: 
\[
\text{Fcn} = \{\, f, u \,\} \cup \{\, u_{pqr} \mid (p, q, r) \in (A \setminus B)^3 \text{ and $p, q, r$ are pairwise distinct} \,\}
\]
$f$ will have arity $n$ and the rest of the function symbols will be unary. We will set $f^\A = \bigcup_{i \in \omega} f_i^{\A_i}$. For proper triples $(p, q, r)$ we will set
\[
u_{pqr}^\A(x) =
\begin{cases}
q & \text{ if } x = p \\
r & \text{ if } x = q \\ 
p & \text{ if } x = r \\
a_{i, j+1} & \text{ if } x = a_{i, j} \\
b_{i, j+1} & \text{ if } x = b_{i, j} \\
x & \text{ otherwise }
\end{cases}
\]
Finally, $\A$ will interpret $u$ as a permutation, written in cycle notation as 
\[
u^\A = (\, a_1 \, b_1 \, a_2 \, b_2 \, \dots \, a_n \, b_n \, c \,)
\]
Note that $f^\A$ is well defined since each $f_{i+1}^{\A_{i+1}}$ is an extension of $f_i^{\A_i}$. Further note that the construction of $\A$ results in a complete algebra.
\end{df}

\begin{lm} \label{nfequal}
If $f^\A(\vp) = f^\A(\vq)$ and $\vp \neq \vq$, then $\vp, \vq \in \dmn(f_0^{\A_0})$.
\end{lm}

\begin{proof}
Let $k$ be the smallest index such that $\vp \in \dmn(f_k^{\A_k})$. Suppose for contradiction $k \neq 0$. Then 
\[
f^\A(\vp) = f_k^{\A_k}(\vp) = (\vp, k-1)
\] 
Now either there exists some $i$ with $f^\A(\vq) = d_i$ or there exists some $j$ with $f^\A(\vq) = (\vq, l)$.
But for all $i$ and for all $j$
\[
(\vp, k-1) \neq d_i \qquad \text{ and } \qquad (\vp, k-1) \neq (\vq, j)
\]
Thus $f^\A(\vp) \neq f^\A(\vq)$, a contradiction. So $k = 0$. Thus $\vp \in \dmn(f_0^{\A_0})$. Switching the roles of $\vp$ and $\vq$ above, we get $\vq \in \dmn(f_0^{\A_0})$ as well. 
\end{proof}

\begin{remark}\label{nproperties}
We list some useful observations about $\A$ here.
	\begin{enumerate}
	\item \label{nuinjective} All of the unary fundamental operations of $\A$ are injective.
	\item \label{nfinjective} $f^\A$ is injective on $A^n \setminus \dmn(f_0^{\A_0})$.
	\item \label{nfrange} The range of $f^\A$ does not intersect $B$.
	\item \label{nurange} $u^\A$ is the only fundamental operation with any element of $C$ in its range.
	\end{enumerate}
\end{remark}




\section{The Theorem}\label{thetheorem}

In this section we prove properties about the algebra constructed in Section \ref{construction} and summarize our main results a theorem.

\begin{prp}\label{nsimple}
$\A$ is simple.
\end{prp}	

\begin{proof}

Suppose $\theta \in \con(\A)$ has $(p, q) \in \theta$ with $p \neq q$. We will show that $\theta = 1_\A$. It will suffice to show that for all $r \in A$ we have $(q, r) \in \theta$. 

It will be advantageous to have that $p$ and $q$ are not elements of $B$. Observe that since $\theta$ is a congruence,
\[
(f^\A(p,p, \dots, p), f^\A(q,q, \dots, q)) \in \theta
\]
By Remark \ref{nproperties} (\ref{nfinjective}), 
\[
f^\A(p,p, \dots, p) \neq f^\A(q,q, \dots, q)
\] 
Finally, by Remark \ref{nproperties} (\ref{nfrange}),
\[
f^\A(p,p, \dots, p), f^\A(q,q, \dots, q) \in A \setminus B
\]
If either $p$ or $q$ is in  $B$, replace $p$ by $f^\A(p,p, \dots, p)$ and replace $q$ by $f^\A(q,q, \dots, q)$. We may assume from now on that $p$ and $q$ are both elements of $A \setminus B$.

Let $r \in A \setminus B$. If $r \in \{\, p, q \,\}$ then since $\theta$ is an equivalence relation, $(q, r) \in \theta$. If $r \notin \{\, p, q \,\}$, then $\A$ has a fundamental operation $u^\A_{p q r}$. Since $(p, q) \in \theta$, we have $(q,r) = (u^\A_{pqr}(p), u^\A_{pqr}(q)) \in \theta$. 

To show that elements of $B$ are $\theta$ related to $q$, first note that $c \in A \setminus B$, so we have that $(q, c) \in \theta$ by the previous paragraph. Thus for all $k \in \omega$ we have $((u^\A)^k(q), (u^\A)^k(c)) \in \theta$. This gives us that $(q, a_i)$ and $(q, b_i)$ are in $\theta$ for all $1 \leq i \leq n$. We are left to show that $(q, a_{i, j})$ and $(q, b_{i, j})$ are in $\theta$ for all $1 \leq i \leq n$ and all $j \geq 1$. Let $p_1, p_2, p_3$  be pairwise distinct elements of $A \setminus (B \cup \{\, q \,\})$. Then for all $j \geq 1$, $(q, a_{i, j}) = ((u^\A_{p_1 p_2 p_3})^j(q), (u^\A_{p_1 p_2 p_3})^j(a_i)) \in \theta$. Similarly for all $j \geq 1$, $(q, b_{i, j})$ is in $\theta$.

We have thus shown for all $r \in A$ that $(q, r) \in \theta$, as desired.
\end{proof}

\begin{prp}\label{nnotnilpotent}
$[\underbrace{1_\A, 1_\A, \dots, 1_\A}_{n \text{ many}}] = 1_\A$
\end{prp}

We prove this proposition first for the case where when $n = 2$. We do this so that the reader may more easily understand the core of the argument and have pictures for reference. The proof for arbitrary $n$ follows a very similar structure to the $n=2$ proof, and in fact works when $n=2$.

\begin{proof}[Proof for $n=2$]
Consider the term $f(x, y)$. Observe in Figure \ref{111} that the term square $S_{f^\A}(a_i, b_i)$ has an inequality on its critical edge. So $[1_\A, 1_\A] \neq 0_\A$. Since $\A$ is simple, $[1_\A, 1_\A] = 1_\A$ as desired.

\begin{figure}[h]
\termsquare{f^\A(a_1,a_2) = d_1}{f^\A(a_1,b_2) = d_1}{d_2 =  f^\A(b_1,a_2)}{d_3 = f^\A(b_1, b_2)}
\caption{$S_{f^\A}(a_i, b_i)$.}
\label{111}
\end{figure}
\end{proof}

\begin{proof}
For the term $f(x_1, \dots, x_n)$ we have
\[
C_{f^\A}^n(a_i, b_i) = (d_1, d_1, d_2, d_2, \dots, d_{2^{n-1}-1}, d_{2^{n-1}-1}, d_{2^{n-1}}, d_{2^{n-1}+1})
\]
Note that
\[
d_{2^{n-1}} \neq d_{2^{n-1}+1}
\]
So $f$ fails the $n$-term condition. Thus 
\[
[\underbrace{1_\A, 1_\A, \dots, 1_\A}_{n \text{ many}}] \neq 0_\A
\] 
Since $\A$ is simple, we must have 
\[
[\underbrace{1_\A, 1_\A, \dots, 1_\A}_{n \text{ many}}] = 1_\A
\]
\end{proof}

\begin{lm}\label{termnotequal}
Let $\tau(\vx)$ be a term in the language of $\A$. If $\tau^\A(\vp) \neq \tau^\A(\vq)$ and $\tau^\A(\vp), \tau^\A(\vq) \in C$, then there is some $i$ and $m \in \omega$ such that $\tau^\A(\vx) = (u^\A)^m(x_i)$. (Note we consider $(u^\A)^0(x) = x$.)
\end{lm}

\begin{proof}
Suppose there is a counterexample to the claim. Let $\tau$ be the shortest such counterexample. I.e. let $\tau$ be a $k$-ary term such that $\tau^\A$ is not a power of $u^\A$ and there are fixed $\vp$ and $\vq$ with $\tau^\A(\vp) \neq \tau^\A(\vq)$ and $\tau^\A(\vp), \tau^\A(\vq) \in C$. 

Note that we know $\tau^\A(\vx) \neq x_i$ for $1 \leq i \leq k$ by assumption. We consider the outer operation of $\tau$. By Remark \ref{nproperties} (\ref{nurange}) we must have some term $\sigma_1$ such that $\tau(\vx) = u(\sigma_1(\vx))$. 

Recall that $u^\A$ is the permutation
\[
u^\A = (\, a_1 \, b_1 \, a_2 \, b_2 \, \dots \, a_n \, b_n \, c \,)
\]
So $u^\A(\sigma^\A_1(\vp)) \neq u^\A(\sigma^\A_2(\vq))$ gives $\sigma_1^\A(\vp) \neq \sigma_1^\A(\vq)$ and $\tau^\A(\vp), \tau^\A(\vq) \in C$ gives $\sigma_1^\A(\vp), \sigma_1^\A(\vq) \in C \cup \{\, c \,\}$. Since $\tau^\A$ is not a power of $u^\A$, neither is $\sigma_1^\A$. Since $\tau$ is the shortest counterexample, we cannot have $\sigma_1^\A(\vp), \sigma_1^\A(\vq) \in C$. Thus we must have one of $\sigma_1^\A(\vp)$ or $\sigma_1^\A(\vq)$ equal to $c$. Relabel $\vp$ and $\vq$ so that $\sigma_1^\A(\vp) = c$. Then $\sigma_1^\A(\vq) \in C$. By Remark \ref{nproperties} (\ref{nurange}) again we must have some term $\sigma_2$ such that $\sigma_1(\vx) = u(\sigma_2(\vx))$. 

Since $\sigma_1^\A(\vp) = c$ we have $\sigma_2^\A(\vp) = b_n$. Since $\sigma_1^\A(\vq) \in C$ we have $\sigma_2^\A(\vq) \in C \cup \{\, c \,\}$. Note that $\sigma_2^\A$ is not a power of $u^\A$ since $\sigma_1^\A$ is not. Again, since $\tau$ is the shortest counterexample we must have $\sigma_2^\A(\vq) = c$. By Remark \ref{nproperties} (\ref{nurange}) once more, since $\sigma_2^\A(\vp) \in C$, there must be some term $\sigma_3$ such that $\sigma_2(\vx) = u(\sigma_3(\vx))$. 

Noting the definition of $u^\A$ again, we observe that $\sigma_3^\A(\vp) = a_n$ and $\sigma_3^\A(\vq) = b_n$. These are both elements of $C$, but $\sigma_3^\A$ is not a power of $u^\A$ since $\sigma_2^\A$ isn't, so we have a shorter counterexample to the statement than $\tau$. This is a contradiction, so no counterexample exists.
\end{proof}

\begin{lm}\label{cornerequal}
Suppose $C^n_{\tau^\A}(\vp_i, \vq_i) = (r_1, r_2, \dots, r_{2^n})$ has its first vertex $r_1$ equal to all of its adjacent vertices. I.e. suppose that 
\[
r_1 = r_2 = r_3 = r_5 = r_9 = \dots = r_{2^{n-1}+1}
\]
Then the cube is constant, that is 
\[
C^n_{\tau^\A}(\vp_i, \vq_i) = (r_1, r_1, \dots, r_1)
\]
\end{lm}

\begin{proof}
We will prove this proposition by contradiction. To that end, suppose $\tau$ is the shortest term in the language of $\A$ such that there exist tuples $\vp_i, \vq_i$ for $1 \leq i \leq n$ such that $C^n_{\tau^\A}(\vp_i, \vq_i) = (r_1, r_2, \dots, r_{2^n})$ has its first vertex equal to all adjacent vertices but $C^n_{\tau^\A}(\vp_i, \vq_i)$ is not constant. 

Suppose that $\tau^\A$ is essentially unary. Then since $C^n_{\tau^\A}(\vp_i, \vq_i)$ is not constant, there is an inequality along an edge. Since $\tau^\A$ is essentially unary, all parallel edges must then also have an inequality. This would mean that one of the vertices adjacent to $r_1$ could not be equal to $r_1$, a contradiction. Thus $\tau^\A$ is not essentially unary. From this we conclude that $\tau$ is not a variable.

Suppose $\tau$ has one of the unary function symbols as its outer operation, that is suppose $\tau = \upsilon (\sigma)$ for some fundamental unary function symbol $\upsilon$ and some term $\sigma$. Then as in Remark \ref{nproperties} (\ref{nuinjective}), $\upsilon^\A$ is injective. Thus we must have that $\sigma$ is a term which is shorter than $\tau$ such that  $C_{\sigma^\A}(\vp_i, \vq_i) = (r_1, r_2, \dots, r_{2^n})$ has its first vertex equal to all adjacent vertices but $C^n_{\sigma^\A}(\vp_i, \vq_i)$ is not constant, contradicting the assumption that $\tau$ is the shortest such term. 

Suppose $\tau$ has outer operation $f$, that is suppose $\tau = f(\sigma_1, \sigma_2, \dots, \sigma_n)$ for terms $\sigma_1, \sigma_2, \dots, \sigma_n$.  Label 
\[
C^n_{\sigma_i^\A} = (s_{i,1}, s_{i,2}, \dots, s_{i,2^n})
\]
Since $C^n_{\tau^\A}$ is not constant, we know that there is some fixed $j$ with $1 \leq j \leq n$ such that $C^n_{\sigma_j^\A}$ is not constant. By the inductive assumption, we know that $C^n_{\sigma_j^\A}$ cannot have its first vertex equal to all adjacent vertices. Thus we may fix $k$ such that $s_{j, k}$ is adjacent to $s_{j, 1}$ and $s_{j, k} \neq s_{j, 1}$. Note that $r_k$ is adjacent to $r_1$ so $r_1 = r_k$. We then have 
\[
f^\A(s_{1, 1}, s_{2, 1}, \dots, s_{n, 1}) = r_1 = r_k = f^\A(s_{1, k}, s_{2, k}, \dots, s_{n, k})
\] 
but
\[
(s_{1, 1}, s_{2, 1}, \dots, s_{n, 1}) \neq (s_{1, k}, s_{2, k}, \dots, s_{n, k})
\]
By Lemma \ref{nfequal} we see that we must have 
\[
(s_{1, 1}, s_{2, 1}, \dots, s_{n, 1}), (s_{1, k}, s_{2, k}, \dots, s_{n, k}) \in \dmn(f_0^{\A_0})
\]
Note that $f_0^{\A_0}(\vx) = f_0^{\A_0}(\vy)$ only when the tuples $\vx$ and $\vy$ agree on their first $n-1$ entries. So we must have $s_{i, 1} = s_{i, k}$ for all $1 \leq i \leq n-1$. So the fact that $s_{j, k} \neq s_{j, 1}$ tells us that $j = n$. Observe that since $(s_{1, 1}, s_{2, 1}, \dots, s_{n, 1}) \in \dmn(f_0^{\A_0})$ there must be some $l$ such that $r_1 = d_l$. We then know that all vertices adjacent to $r_1$ in $C^n_{\tau^\A}$ are also equal to $d_l$. This tells us that for all $1 \leq i \leq n-1$ any vertex adjacent to $s_{i, 1}$ in $C^n_{\sigma_i^\A}$ is in fact equal to $s_{i, 1}$. So by the inductive assumption, we have that $C^n_{\sigma_i^\A}$ is constant for all $1 \leq i \leq n-1$. We consider $C^n_{\sigma_n^\A}$. We know that $s_{n, 1} \neq s_{n, k}$ and since $(s_{1, 1}, s_{2, 1}, \dots, s_{n, 1}), (s_{1, k}, s_{2, k}, \dots, s_{n, k}) \in \dmn(f_0^{\A_0})$ we further have that $s_{n, 1}, s_{n, j} \in C$. From this Lemma \ref{termnotequal} gives us that $\sigma_n^\A = (u^\A)^m$ for some $m$. Important in this fact is that $\sigma_n^\A$ is essentially unary and thus all vertices of $C^n_{\sigma_n^\A}$ are either $s_{n, 1}$ or $s_{n, j}$. But this means that $C^n_{\tau^\A}$ is in fact the constant cube with all vertices $d_l$, a contradiction. 

We have thus found that $\tau$ may not be a variable, and may not have any of the fundamental function symbols as its outer operation, thus no such $\tau$ may exist, as desired. 
\end{proof}

\begin{prp}\label{supernilpotent}
$[\underbrace{1_\A, 1_\A, \dots, 1_\A}_{n+1 \text{ many}}] = 0_\A$
\end{prp}

For similar reasons to those in Proposition \ref{nnotnilpotent}, we first prove this proposition when $n=2$, then follow with the proof for the general case.

\begin{proof}[Proof for $n = 2$]
We will prove this proposition by contradiction. To that end, suppose 
\[
[1_\A, 1_\A, 1_\A] \neq 0_\A
\]
Let $\tau(\vx_1, \vx_2, \vx_3)$ be the shortest term witnessing the failure of the 3-dimensional term condition. 

If $\tau^\A$ were essentially unary, then the inequality along the critical edge of $C_{\tau^\A}$ would imply inequality along the three bold vertical edges (see Figure \ref{taufail}) contradicting $\tau$ being a witness to the failure of the 3-dimensional term condition. Thus $\tau^\A$ cannot be essentially unary. So $\tau$ is not a variable.

Suppose there is a unary function symbol $\upsilon$ and term $\sigma$ such that 
\[
\tau(\vx_1, \vx_2, \vx_3) = \upsilon(\sigma(\vx_1, \vx_2, \vx_3))
\]
Noting that $\upsilon^\A$ is injective, as in Remark \ref{nproperties} (\ref{nuinjective}), we see that $\sigma$ will be a shorter term witnessing the failure of the 3-dimensional term condition, a contradiction.

Suppose there are terms $\sigma_1, \sigma_2$ such that 
\[
\tau(\vx_1, \vx_2, \vx_3) = f(\sigma_1(\vx_1, \vx_2, \vx_3), \sigma_2(\vx_1, \vx_2, \vx_3))
\]
Let $C_{\tau^\A}(\vp_i, \vq_i) = (r_1, r_2, \dots, r_8)$ witness the failure of the 3-dimensional term condition, as in Figure \ref{taufail}.
\begin{figure}[h]
\termcube{r_1}{r_2}{r_3}{r_4}{r_5}{r_6}{r_7}{r_8}
\caption{$C_{\tau^\A}(\vp_i, \vq_i)$}
\label{taufail}
\end{figure}

Label the term cubes $C_{\sigma_1^\A}(\vp_i, \vq_i)$ and $C_{\sigma_2^\A}(\vp_i, \vq_i)$ as in Figure \ref{sigma12}.
\begin{figure}[h]
$C_{\sigma_1^\A}$: \tcube{s_1}{s_2}{s_3}{s_4}{s_5}{s_6}{s_7}{s_8}  \hskip1cm $C_{\sigma_2^\A}$: \tcube{t_1}{t_2}{t_3}{t_4}{t_5}{t_6}{t_7}{t_8} 
\caption{$C_{\sigma_1^\A}(\vp_i, \vq_i)$ on the left. $C_{\sigma_2^\A}(\vp_i, \vq_i)$ on the right.} 
\label{sigma12}
\end{figure}
Since 
\[
f(s_7, t_7) = r_7 \neq r_8 = f(s_8, t_8)
\]
we know that
\[
(s_7, t_7) \neq (s_8, t_8)
\]
So $s_7 \neq s_8$ or $t_7 \neq t_8$. Since $\sigma_1$ and $\sigma_2$ are both shorter terms than $\tau$, we know that neither witnesses the failure of the 3-dimensional term condition. Thus there must be a failure in equality along a non-critical vertical edge in $C_{\sigma_1^\A}(\vp_i, \vq_i)$ or $C_{\sigma_2^\A}(\vp_i, \vq_i)$, respectively. I.e. there is some $j$ with $1 \leq j \leq 3$ such that 
\[
s_{2j-1} \neq s_{2j} \text{ or } t_{2j-1} \neq t_{2j}
\]
In either case
\[
(s_{2j-1}, t_{2j-1}) \neq (s_{2j}, t_{2j})
\]
Observe that
\[
f^\A(s_{2j-1}, t_{2j-1}) = r_{2j-1} = r_{2j} = f^\A(s_{2j}, t_{2j})
\]
By Lemma \ref{nfequal} we get that 
\[
(s_{2j-1}, t_{2j-1}), (s_{2j}, t_{2j}) \in \dmn(f_0^{\A_0})
\]
Thus 
\[
s_{2j-1}, s_{2j}, t_{2j-1}, t_{2j} \in C
\]
and, in fact, $s_{2j-1} = s_{2j},$ while $t_{2j-1} \neq t_{2j}$. So $t_{2j-1}$ and $t_{2j}$ are distinct elements of $C$ in the range of $\sigma_2^\A$. Lemma \ref{termnotequal} then tells us that there is some $m \in \omega$ such that $\sigma_2^\A = (u^\A)^{m}$. So $\sigma_2^\A$ is essentially unary. From this we may conclude that 
\[
C_{\sigma_2^\A}(\vp_i, \vq_i) = (t_1, t_2, t_1, t_2, t_1, t_2, t_1, t_2)
\]
with $t_1, t_2$ distinct elements in $\{a_2, b_2\}$. Further, since $r_1 = r_2$, $r_3 = r_4$, and $r_5 = r_6$ we must have 
\[
s_1 = s_2 = s_3 = s_4 = s_5 = s_6 = a_1
\]
Note that both $(s_1, s_3, s_5, s_7)$ and $(s_2, s_4, s_6, s_8)$ are term squares for $\sigma_1$. By Lemma \ref{cornerequal} we get that both term squares are constant. So $s_7 = s_8 = a_1$ as well. We may now explicitly compute that $C_{\tau^\A}(\vp_i, \vq_i)$ is a constant cube. This contradicts our assumptions and we have completed the proof. 
\end{proof}

\begin{proof}
We will prove this proposition by contradiction. To that end, suppose 
\[
[\underbrace{1_\A, 1_\A, \dots, 1_\A}_{n+1 \text{ many}}] \neq 0_\A
\]
Let $\tau(\vx_1, \dots, \vx_{n+1})$ be the shortest term witnessing the failure of the $n+1$-dimensional term condition. 

If $\tau^\A$ were essentially unary, then the inequality along the critical edge of $C^{n+1}_{\tau^\A}$ would imply inequality along all other vertical edges, contradicting $\tau$ being a witness to the failure of the $n+1$-dimensional term condition. Thus $\tau^\A$ cannot be essentially unary. So $\tau$ is not a variable.

Suppose there is a unary function symbol $\upsilon$ and term $\sigma$ such that 
\[
\tau(\vx_1, \dots, \vx_{n+1}) = \upsilon(\sigma(\vx_1, \dots, \vx_{n+1}))
\]
Noting that $\upsilon^\A$ is injective, as in Remark \ref{nproperties} (\ref{nuinjective}), we see that $\sigma$ will be a shorter term witnessing the failure of the $n+1$-dimensional term condition, a contradiction.

Suppose there are terms $\sigma_1, \dots, \sigma_n$ such that 
\[
\tau(\vx_1, \dots, \vx_{n+1}) = f(\sigma_1(\vx_1, \dots, \vx_{n+1}), \dots, \sigma_n(\vx_1, \dots, \vx_{n+1}))
\]
Let
\[
C_{\tau^\A}^{n+1}(\vp_i, \vq_i) = (r_1, r_2, \dots, r_{2^{n+1}})
\]
witness the failure of the $n+1$-dimensional term condition, so that 
\[
r_{2j-1} = r_{2j} \qquad \text{ for all } 1 \leq j \leq 2^{n} - 1
\]
and
\[
r_{2^{n+1}-1} \neq r_{2^{n+1}}
\]
For each $1 \leq i \leq n$, label the $n+1$-term cubes as 
\[
C_{\sigma_i^\A}^{n+1}(\vp_i, \vq_i) = (s_{i, 1}, s_{i, 2}, \dots, s_{i, 2^{n+1}})
\]
Since $r_{2^{n+1}-1} \neq r_{2^{n+1}}$, there must be some fixed $k$ with $1 \leq k \leq n$ such that $s_{k, 2^{n+1}-1} \neq s_{k, 2^{n+1}}$. Since $\sigma_k$ is a shorter term than $\tau$, we know that $\sigma_k$ does not witness the failure of the $n+1$-dimensional term condition. So there must be some fixed $l$ with $1 \leq l \leq 2^n-1$ such that
\[
s_{k, 2l-1} \neq s_{k, 2l} 
\]
Thus 
\[
(s_{1, 2l-1}, s_{2, 2l-1}, \dots, s_{n, 2l-1}) \neq (s_{1, 2l}, s_{2, 2l}, \dots, s_{n, 2l})
\]
Observe that
\[
f^\A(s_{1, 2l-1}, s_{2, 2l-1}, \dots, s_{n, 2l-1}) = r_{2l-1} = r_{2l} = f^\A(s_{1, 2l}, s_{2, 2l}, \dots, s_{n, 2l})
\]
By Lemma \ref{nfequal} we get that 
\[
(s_{1, 2l-1}, s_{2, 2l-1}, \dots, s_{n, 2l-1}), (s_{1, 2l}, s_{2, 2l}, \dots, s_{n, 2l}) \in \dmn(f_0^{\A_0})
\]
Thus for each $1 \leq i \leq n$ we have
\[
s_{i, 2l-1}, s_{i, 2l} \in C
\]
and, in fact, for each $1 \leq i <n$, we have $s_{i, 2l-1} = s_{i, 2l},$ while $s_{n, 2l-1} \neq s_{n, 2l}$. So $s_{n, 2l-1}$ and $s_{n, 2l}$ are distinct elements of $C$ in the range of $\sigma_n^\A$. Recall that $s_{n, 2l-1}$ and $s_{n, 2l}$ are in the range of $\sigma_n^\A$. Lemma \ref{termnotequal} then tells us that there is some $m \in \omega$ such that $\sigma_n^\A = (u^\A)^{m}$. So $\sigma_n^\A$ is essentially unary. From this we may conclude that 
\[
C^{n+1}_{\sigma_n^\A}(\vp_i, \vq_i) = (s_{n, 1}, s_{n, 2}, s_{n, 1}, s_{n, 2}, \dots , s_{n, 1}, s_{n, 2})
\]
with $s_{n, 1}, s_{n, 2}$ distinct elements in $\{a_n, b_n\}$. Since for all $1 \leq j \leq 2^{n} - 1$ we know $r_{2j-1} = r_{2j}$, we must have 
\[
s_{i, 2j-1} = s_{i, 2j} \in \{a_i, b_i\} \qquad \text{ for all } 1 \leq j \leq 2^{n} - 1 \text{ and all } 1 \leq i < n
\]
For each cube $C^{n+1}_{\sigma_i^\A}$ we have two possible cases.

\begin{case} For all $j_1$ and $j_2$ with $1 \leq j_1 \leq j_2 \leq 2^{n+1}-2$ we have $s_{i, j_1} = s_{i, j_2}$. \end{case}

In this case we see that Lemma \ref{cornerequal} applies and in fact $C^{n+1}_{\sigma_i^\A}$ is the constant cube with vertices all $a_i$ or all $b_i$.

\begin{case} There are $j_1$ and $j_2$ with $1 \leq j_1 \leq j_2 \leq 2^{n+1}-2$ such that $s_{i, j_1} \neq s_{i, j_2}$. \end{case}

In this case we have that $\sigma_i^\A$ outputs two distinct elements of $C$, so Lemma \ref{termnotequal} tells us that $\sigma_i^\A$ is actually a power of $u^\A$, so is essentially unary. From this and the fact that $s_{i, 2j-1} = s_{i, 2j} \in \{a_i, b_i\}$ for all $1 \leq j \leq 2^{n} - 1$, we may conclude that $s_{i, 2^{n+1} -1} = s_{i, 2^{n+1}} \in \{a_i, b_i\}$. 

Note that in either case, $s_{i, 2^{n+1} -1} = s_{i, 2^{n+1}} \in \{a_i, b_i\}$. We may now explicitly compute that $C^{n+1}_{\tau^\A}(\vp_i, \vq_i)$ is a constant cube. This contradicts the assumption that we chose $C^{n+1}_{\tau^\A}(\vp_i, \vq_i)$ to be a witness of the failure of the $n+1$-term condition.
\end{proof}

The three propositions of this section immediately entail the following:

\begin{thm}
There exists a simple algebra $\A$ whose descending central series of higher commutators, $(\theta_{m+1})_{m \in \omega}$, has 
\[
\theta_m = 1_\A \text{ for } m \leq n
\]
and 
\[
\theta_m = 0_\A \text{ for } m > n
\]
\end{thm}




\section{Acknowledgments}

This paper was inspired by earlier work on representing the binary commutator in \cite{kearnes-lampe-willard} by Keith A. Kearnes, William A. Lampe, and Ross Willard.


\end{document}